\documentclass[12pt]{amsart}
\usepackage{amsmath,amsfonts,amssymb}
\usepackage{amsmath}
\usepackage{amsfonts}
\usepackage{amssymb}

\numberwithin{equation}{section}

\date{\today}



\newcommand\MM{{\mathcal M}}
\newcommand\AAA{{\mathcal A}}
\newcommand\NN{{\mathcal N}}

\newcommand{\mathdef}[1]{{\bf{#1}}}
\newcommand{\abs}[1]{{\left|#1\right|}}
\newcommand{\bra}[1]{{\langle#1\rangle}}

\newcommand\CC{{\mathbb{C}}}
\newcommand\QQ{{\mathbb{Q}}}
\newcommand\RR{{\mathbb{R}}}
\newcommand\ZZ{{\mathbb{Z}}}

\DeclareMathOperator{\Alt}{Alt}
\DeclareMathOperator{\Br}{Br}
\DeclareMathOperator{\End}{End}
\DeclareMathOperator{\Gal}{Gal}

\DeclareMathOperator{\Hom}{Hom}
\DeclareMathOperator{\inv}{inv}
\DeclareMathOperator{\Irr}{Irr}

\DeclareMathOperator{\Mat}{Mat}

\DeclareMathOperator{\Sym}{Sym}

\newcommand\Sn{{\Sym}_n}
\newcommand\An{{\Alt}_n}

\newtheorem{theorem}{Theorem}[section]

\newtheorem{question}[theorem]{Question}

\newtheorem{lemma}[theorem]{Lemma}
\newtheorem{corollary}[theorem]{Corollary}

\begin{document}

\title{Certain subgroups of Weyl groups are split}

\author{Daniel Goldstein} 
\address{
Center for Communications Research, 
4320 Westerra Court,
San Diego, 
CA 92121-1967} \email{dgoldste@ccrwest.org}

\author{\hbox{Robert M.~Guralnick}} 
\address{
Department of Mathematics,
University of Southern California,
CA 90089-2532} \email{guralnic@usc.edu}

\begin{abstract} \label{abs}
Let $C$ be the centralizer in a finite Weyl group
of an elementary abelian $2$-subgroup.
We show that every complex representation of
$C$ can be realized over the field of rational numbers.
The same holds for a Sylow $2$-subgroup of~$C$.
\end{abstract}

\maketitle

\tableofcontents

\section{Introduction} \label{intro}

The purpose of this note is to generalize the
well-known result that every complex representation of a finite
Weyl group can be realized over the field of rational number $\QQ$.
Benard \cite{benard} proved this by a case-by-case analysis of the
exceptional groups. Springer's 
cohomological construction \cite[Corollary 1.5]{springer}
gives a more uniform proof.

We set some notation that is used throughout the paper.
$W$ is a finite Weyl group, 
$E$ is an elementary abelian $2$-subgroup of $W$,
$C$ is the centralizer of $E$ in $W$, and $T$ is 
a Sylow $2$-subgroup of~$C$.

\begin{theorem} \label{main} 
All complex representations of the groups $C$ and $T$ can
be realized over $\QQ$.
\end{theorem}

Theorem~\ref{main} reduces to the result of Benard and Springer
in the case $E=1$. 
The fact that every irreducible complex character of 
$C$ has values in $\QQ$ is proved
in~\cite{kol}. The case of classical Weyl groups 
was treated in \cite{gm}.

We say the finite group $G$ is \mathdef{split}
if each of its complex representations can be realized over $\QQ$.
With this terminology, Theorem~\ref{main} is the assertion 
that $C$ and $T$ are split.
Clearly, $G$ is split if and only if each of its 
irreducible complex representations can be realized over $\QQ$
(since every complex representation of $G$ is a direct sum of
irreducible representations).
Let $\Irr(G)$ denote the set of irreducible complex
characters of $G$.
We say that $\chi\in \Irr(G)$ is
\mathdef{split} if it is the trace of a $\QQ G$-module.
If $\chi$ is split, then $\chi$ is rational-valued, but
the converse can fail (e.g. the quaternion group of order eight).
In fact, $\chi$ is split if and only if $\chi$ is rational-valued and
the Schur index of $\chi$ is equal to $1$.

In this paper, we use both the Schur index and the Frobenius-Schur
indicator. We remark that for a $2$-group, the Schur index of 
an irreducible character is trivial if and only if the Schur index is 
$+1$ (Corollary~\ref{2-groups}).

We mention some easy reductions used in the proof.
It suffices to prove the theorem for 
irreducible finite Weyl groups,
as any finite Weyl group is a direct product of these
(see Section~\ref{direct products}).

An irreducible Weyl group is either of classical or exceptional type.
The proof for a classical Weyl group $W$ is in
Section~\ref{classical}.

For the exceptional Weyl groups, we prove a slightly stronger statement 
(Theorem~\ref{gg3}). We suspect (but do not prove) that
this stronger statement also holds for the classical case.

Let $G$ be a finite group and $S$ a Sylow $2$-subgroup of $G$.
Then the following four conditions, if they hold,
would imply that both $G$ and $S$ are split.

\begin{enumerate}

\item Each irreducible complex character of $S$ is rational-valued.

\item Each Frobenius-Schur indicator of $S$ is equal to $+1$.

\item Each irreducible complex character $\psi$ of $G$
occurs with odd multiplicity in an induced
character from~$S$.

\item Each irreducible complex character $\psi$ of $G$ 
is rational-valued.
\end{enumerate}

By the remark, the $2$-group 
$S$ is split if and only if conditions (1) and~(2) hold. 

Now assume that $S$ is a Sylow $2$-subgroup with
$S$ split. 

Let $\psi$ be an irreducible complex character of $G$
such that conditions  (3) and (4)  hold.
It follows from the theory of Schur indices and the Brauer-Speiser theorem
(see \cite{brauer}, \cite{speiser})
that $\psi$ is the trace of a $\QQ G$-module (Corollary~\ref{useful}).
Consequently, conditions (3) and (4) , for all such $\psi$, would 
imply that $G$ is split.

We show that conditions (1)--(4) hold for every exceptional 
($C$, $T$) pair.
We note that our proof of Theorem~\ref{main} 
for the classical groups does not use conditions (1)--(4).
We do not know whether condition~(3) 
holds for the classical groups.

Our proof of this stronger statement for an exceptional Weyl group~$W$,
relies heavily on a
computer calculation. Our method is as follows.
First, we enumerate the conjugacy classes of
subgroups $C\le W$ that are centralizers of an elementary
abelian $2$-subgroup. 
This list can be computed by an easy downward
induction. (Choose an involution. Take its centralizer.
Remove duplicate subgroups. Repeat.
See Appendix~\ref{code} for Magma code~\cite{magma} to do this.)

There is a good reason for using this sort of algorithm. Namely, 
the number of
conjugacy classes of elementary
$2$-subgroups of $W$
can be much
larger than the number of 
conjugacy classes of 
centralizers of elementary abelian $2$-subgroups of $W$. 
For the symmetric groups, see
Section~\ref{symmetric}.

\section{Symmetric groups} \label{symmetric}

Let $\Sn$ denote the symmetric group on $n$ letters.
Let $e(n)$ denote the number of conjugacy classes of elementary
abelian $2$-subgroups of $\Sn$. Let $c(n)$ denote the number of conjugacy
classes of centralizers of elementary abelian $2$-subgroups of $\Sn$.
In this section we show that $e(n)\ge c(n)$,
at least for large $n$.

Let $\log$ denote the logarithm to the base $2$.

\begin{lemma} \label{counting}
Let $n\ge1$. Then
\begin{enumerate}
\item $\log c(n) \le (\log n)^2$, and 
\item $\log e(n) \ge  (n-3)^2/16 - n/4 - n\log n $.
\end{enumerate}
\end{lemma}

Let $E$ be an elementary abelian $2$-subgroup of $\Sn$.
Each orbit of $E$ has size a power of~$2$.
Let $a_k$ be the number of orbits of size~$2^k$.
Let~$C$ be the centralizer of~$E$. 
Then~$C$ is direct product of groups $C_k$, where each $C_k$ 
is a wreath product of the symmetric group $\Sym_{a_k}$
with an elementary abelian $2$-group of order $2^k$ acting regularly
on $2^k$ points.

It follows from this description that $c(n)$ 
is equal to the number of partitions
of $n$ into powers of $2$. In other words, $c(n)$ is equal to the
number of ways of writing 
$$n = a_0\cdot1 + a_1\cdot2 + a_2\cdot4 + \cdots + a_k\cdot2^k
$$
with nonnegative integers $a_0,\dots, a_k$.

Since $k\le \log n$, and each $a_i\le n$, it follows that
$c(n)\le n^{\log n}$.

Let $V$ be a vector space of dimension $2m$ over the field of $2$
elements. Then the number of subspaces of $V$ of dimension $m$ is 
\begin{align*} 
\frac{(2^{2m}-1) (2^{2m}-2)\cdots (2^{2m} - 2^{m-1})}
{(2^{m}-1) (2^{m}-2)\cdots (2^{m} - 2^{m-1})}
& \ge \frac{2^{(2m-1)m}}{2^{m^2}} = 2^{m^2-m}.
\end{align*}

Next we estimate the number of conjugacy classes of elementary abelian
$2$-subgroups of $\Sn$. Set $m=\lfloor n/4 \rfloor$.
Let $E$ be an elementary abelian $2$-group of rank $2m$ 
acting on $4m$ points, with $2m$ orbits of size~$2$.
By the previous paragraph, the number of subgroups of $E$ of order
$2^m$ is at least $2^{m^2-m}$.
Since each conjugacy class of subgroups can have at most $n!$
members, we have $e(n) \ge 2^{m^2}/(2^{m} n!)$.
Since $n!\le n^n,$ we see that $\log e(n) \ge (n-3)^2/16 - n/4 - n\log
n$.

This finishes the proof of the lemma.
From the lemma and a direct calculation, one sees that $e(n)\ge
c(n)$ for $n\ge 129$.

\section{Direct products} \label{direct products}

By the following lemma, whose straightforward proof we omit,
it suffices to prove Theorem~\ref{main} for 
irreducible finite Weyl groups.

Consider the property $P(G)$ of a group $G$:

\smallskip

\begin{quotation} 
\noindent $P(G):$ For each elementary abelian $2$-subgroup  $E$ of $G$,
the centralizer of $E$ in $G$ and each of its Sylow $2$-subgroups are
split.
\end{quotation}

\smallskip 

\begin{lemma} \label{products}
Let $A=\prod A_i$ be a finite direct product of finite groups. Then
\begin{enumerate}
\item $A$ is split if and only if each $A_i$ is split.
\item $P(A)$ holds if and only if
  $P(A_i)$ holds for each $i$.
\end{enumerate}
\end{lemma}

\section{Wreath products} \label{wreath}
Let $H$ be a finite group.
Let $H_n$ denote the wreath product $H \wr \Sn$.
We show that for each $n>0, H$ is split if and only if
$H_n$ is split.

First we construct some $\QQ H_n$-modules.
Given a $\QQ H$-module $M$, let $\Sym_n$ act on
the $n$-fold tensor product $M^{\otimes n}$
by permuting coordinates.
As this action is compatible with the $\QQ H^n$-module structure,
$M^{\otimes n}$ acquires a $\QQ H_n$-module structure.

Finally, for any $\QQ H$-module $M$ and any $\QQ\Sym_n$-module $N$,
we define the $\QQ H_n$-module $W(M,N)= M^{\otimes n} \otimes N$.
On the first factor, $H_n$ acts as described in the previous
paragraph.
On the second factor, $H_n$ acts through its quotient $\Sym_n=H_n/H^n$.

Its dimension is $\dim W(M,N) = (\dim M)^{n}\cdot \dim N$.
If $M$ and $N$ are irreducible, then $W(M,N)$ is irreducible.

Let $M_1,\dots,M_k$, be representatives for the isomorphism classes of
irreducible $\QQ H$-modules. Set $d_i= \dim M_i$.

Let $\AAA$ denote the set of tuples $a = (a_1,\dots,a_k)$ of nonnegative
integers with sum $a_1+\dots +a_k=n$.
For $a= (a_1,\dots,a_k) \in \AAA$, set
$P_a= \prod_{1\le i\le k}\Sym_{a_i}$.
Then

\begin{equation} \label{w1}
\abs{\Sn: P_a} = \frac{n!}{a_1!\cdots a_k!}.
\end{equation}

For each $a\in\AAA$, 
let $\NN_a$ denote the set of irreducible $\QQ P_a$-modules.
($\NN_a$ may be identified (via tensor product)
with the set of tuples $(N_1,\dots, N_k)$,
where each $N_i$ is an irreducible $\QQ \Sym_{a_i}$-module).

For $a\in \AAA$ and $\nu\in\NN_a$,
form the $H_n.P_a$-module
\begin{equation} \label{bigguy}
\bigotimes_{1\le i\le k} W(M_i,N_i),
\end{equation}
and let $Z(a,\nu)$ be the $H_n$-module obtained by induction.
The $\QQ H_n$-modules $Z(a,\nu)$ are irreducible and pairwise nonisomorphic.
We have

\begin{eqnarray} \label{w2}
\dim W(a,\nu) &= &|\Sn:P_a|\ \dim \nu, \quad\mathrm{and}\\
\label{w3}
\sum_{\nu\in N_a}(\dim \nu)^2 &= &  \abs{P_a}
\end{eqnarray}

Now we have all the ingredients necessary to prove our result on
wreath products.

\begin{theorem} \label{wreath2}
Let $H$ be a finite group. Let $n\ge1$ be an integer.
Then $H$ is split
if and only if $H_n=H\wr\Sym_n$ is split.
\end{theorem}

\begin{proof}
In one direction, assume that $H_n$ is split and let 
$V$ be an irreducible  $\QQ H$-module.
Form the irreducible $\QQ H_n$-module $W(V,\rm{triv}) = V^{\otimes n}$.
Since $W(V,\rm{triv})$ is absolutely irreducible, so is $V$.

Conversely, suppose that $H$ is split.
Enumerate as above the irreducible $\QQ H$-modules.
We have
\begin{eqnarray*}
\sum_a\sum_{\nu\in N_a} \dim Z(a,\nu)^2 &= & \sum_a \sum_{\nu}
(|\Sn:P_a|\ \dim \nu)^2\\
&=&  \sum_a |\Sn:P_a|^2\sum_{\nu}(\dim \nu)^2\\
&=&  \abs{\Sn} \sum_a |\Sn:P_a|\\
&=& n!\sum_a  \frac{n!}{a_1!\cdots a_k!}\\
&= & n!k^n,
\end{eqnarray*}
where the last equality follows from the multinomial formula.
Since $n!k^n = \abs{H_n}$,
we see that every irreducible complex representations of $H_n$
is equal to one of the $Z(z,\nu)$.
Since these are all split, the theorem is proved.
\end{proof}

\begin{corollary} \label{sylowwreath}
Let $G$ be a finite group that is split.
Let $m\ge1$ be a power of $2$, $S$ a Sylow 2-subgroup of $\Sym_m$.
Let $G \wr S$ be the wreath product given by the
permutation action coming from $S\le \Sym_m$.
Then  $G \wr S$ is split.
\end{corollary}

\begin{proof} Set $m=2^r$. Then $r\ge0$.
If $r=0$ there is nothing to show.
If $r= 1$ then $\Sym_m$ is a $2$-group,
so the result follows from Theorem~\ref{wreath2}.
We proceed by induction on $r>1$. Then,
$S = T \wr (\ZZ/2\ZZ)$, where $T$ is a Sylow $2$-subgroup of
$\Sym_{2^{r-1}}$.
Thus $G \wr S \cong (G \wr T) \wr (\ZZ/2\ZZ)$. The result follows by
induction and the case $r=1$.
\end{proof}

\section{Division algebras} \label{division}
We record some known results about division algebras that
are needed for Section~\ref{indices} on Schur indices.

Let $K$ be field.
The collection $\Br(K)$ of all (isomorphism classes of) division algebras
that are finite-dimensional and central over $K$ has the
structure of additive group. The identity element is $K$ itself and
the inverse of $D$ is $D^{\mathrm{opp}}$, the opposite algebra.
This group law can be characterized be the rule
$D+D'=D''$ whenever
\begin{equation} \label{braueraddition}
D\otimes D'\cong \Mat_n(D'')
\end{equation}
for some $n$.
(Such a $D''$ and an $n$ always exist.)

If $L$ is an extension of $K$, then there is a map $e_{L/K}:\Br(K)\to
\Br(L)$ characterized by the property: $e_{L/K}(D)=D'$ if
$D\otimes_KL=M_n(D')$ for some $n$. (Such a $D'$ and an $n$ always
exist.)

Suppose $K$ is a finite extension of $\QQ$.
Let $\MM$ denote the set of places of $K$,
including the infinite places.
Let $K_v$ denote the completion of $K$ at the place $v\in\MM$.
Let $k$ be one of the $K_v$'s.

It is known that $\Br(k) = \QQ/\ZZ$ if $k$ is
non-Archimedean, $\Br(k)=\frac12\ZZ/\ZZ$ if $k$ is the real field, and
$\Br(k)=\{0\}$ if $k$ is the complex field.

Let $D$ be a finite-dimensional division algebra with center $K$.
The local-global theory of division algebras over a number field
assigns to each place $v$ in $\MM$ a
local invariant $\inv_v(D)\in \Br(K_v)$.

The local invariants enjoy the following properties.
\begin{enumerate}

\item For each $D$, $\inv_v(D)$ is nonzero
for at most finitely many $v\in \MM$.

\item For each $D$,
the sum $\sum_{v\in \MM} \inv_v(D)$, which is well-defined by~(1),
is equal to $0$.

\item Let $(\epsilon_v)_{v\in \MM}$ be any collection of local
invariants $\epsilon_v\in \Br(K_v)$ that satisfy (1) and~(2).
Then there exists a $D$ such that
$\inv_v(D) = \epsilon_v$ for all $v\in\MM$.

\item If $\epsilon_v(D) = \epsilon_v(D')$ for all $v\in\MM$, then
$D$ is isomorphic to $D'$.
\end{enumerate}

Suppose $|L:K|<\infty$. The local invariants of $D$ and
of $e_{L/K}(D)$ are related as follows.
Let $w$ be a place of $L$ lying over the place $v$ of
$K$, and let $L_w$ denote the completion of $L$ at $w$.
Then $\inv_w(e_{L/K}(D)) = \abs{L_w:K_v}\inv_v(D)$.

Let now $D$ be a finite-dimensional division algebra $D$ with center
$F$, $F$ an arbitrary field. There exists a maximal commutative
subalgebra $A$ of $D$. The dimension of such an $A$ is called
the \mathdef{degree} of $D$. This does not depend on the choice of $A$,
since $(\dim_F A)^2= \dim_F D$. The rank of $F$ itself is $1$, and
a division algebra of rank~$2$
is a \mathdef{quaternion algebra}.

\begin{lemma} \label{ind-exp}
Let $D$ be a finite-dimensional division algebra with center $K$,
a finite extension of $\QQ$.
Then $D\cong D^{\mathrm{opp}}$ if and
only if $D$ has order $\le2$ in the Brauer group $\Br(K)$.
If this condition holds, then $D$ is a
quaternion algebra or $D=K$.
\end{lemma}
\begin{proof}
The first statement follows from the definition of inversion in the
Brauer group.
The condition holds if and only if each local invariant is equal to
$0$ or $1/2$. Let $(\epsilon_v)_{v\in \MM}$ be a collection of local invariants
$\epsilon_v\in \Br(K_v)$ that satisfy (1) and (2) and such
that for all $v$, $\epsilon_v\in \frac12\ZZ/\ZZ$.
We may assume that not all the local invariants are zero.

Let $S$ be the set of places of $K$ such that $\epsilon_v$ is nonzero.
By (1) and (2), $S$ is a finite set of even cardinality, and
if $v\in S$ then $K_v$ is not the complex field.
By weak approximation, there is $x$ in $K$ whose image in $K_v$ is a
nonsquare for each $v\in S$.
Let $L=K(\sqrt{x}).$ Then $|L_w:K_v|=2$ for each $v\in S$,
where $L_w$ is the completion of $L$ at a place $w$ lying over $v$.
It follows that all the local invariants of $D\otimes_K L$ are
trivial. By (3) applied to $L$, we see
that $D\otimes_K L\cong M_n(L)$ for some $n$.
If $L$ is a splitting field for $D$, then $\dim D\le |L:K|^2.$
\end{proof}

In the case $K=\QQ$ one can be much more explicit.
Now let $\MM$ be the set of places of $\QQ$.
Let $(\epsilon_v)_{v\in\MM}$ be a collection of local invariants
$\epsilon_v\in\frac12\ZZ/\ZZ$, at most finitely
many of which are nonzero, such that $\sum_{v\in\MM}\epsilon_v=0$.
One can show directly that a quaternion algebra over $\QQ$ exists
with local invariants $(\epsilon_v)$ by appealing to Dirichlet's theorem.

For $a,b$ nonzero elements of $\QQ$, one constructs the quaternion algebra
$(a,b)$. Its local invariant at the place $v$ is easily
calculated using the Hilbert symbol
\begin{equation} \label{hilbertsymbol}
\epsilon_v((a,b))= \bra{a,b}_v.
\end{equation}
A good reference for the Hilbert symbol is \cite[Chapter III]{serre}.

Let $S$ be the set of places $v$ of $\QQ$ such that $\epsilon_v$ is nonzero.
We assume $S$ contains $2$ and the infinite place,
and leave the other cases, which are treated similarly, to the reader.

Set $T = S\setminus\{2,\infty\}$.
Choose for each prime $p\in T$ a nonsquare $a_p$ in
$\ZZ/p\ZZ$.
By Dirichlet's theorem and the Chinese remainder theorem, there exists
a positive prime number $q$ such that
\begin{eqnarray}
q  & \equiv & -a_p \pmod p \qquad\mbox{for all $p\in T$, and}\\
q & \equiv & 3 \pmod 8.
\end{eqnarray}

Let $t=\prod_{p\in T} p.$
Now the quaternion algebra $D' = (-2t,-q)$ does the trick.
We have $\inv_{\infty}(D')=\bra{-2t,-q}_{\infty}$ is nonzero, since $t,q>0$.
For $p\in T$,  we have $\inv_{p}(D') = \bra{-2t,a_p}=1/2$.
At the prime $2$, we have $\inv_2(D')=\bra{-2t,-q}_2 =  \bra{-2t,5}_2
=\bra{2,5}_2\bra{-t,5}_2=1/2$.
Clearly, $\inv_v(D')=0$ for places $v$ not in $S\cup \{q\}$.
We have just seen that $D'$ has invariant $1/2$ at each place $v\in S$.
Therefore, since the sum of the local invariants
is $0$, and $\abs{S}$ is even, it follows that $\inv_q(D')=0.$

Finally, it follows by (4) that $D'\cong D$ for any division algebra
$D$ with local invariants $\inv_v(D)=\epsilon_v$.

\section{Schur indices} \label{indices}
The relation to Schur indices is as follows.
Let $\chi$ be an irreducible complex character of the finite group
$G$.
Let $K$ be an extension of $\QQ$ that contains the values of $\chi$.

The \mathdef{Schur index} of $\chi$ over $K$
is the least positive integer $m_{\chi}$ such
that $m_{\chi}\chi$ is the trace of a $K G$-module.
By definition, $m_{\chi}=1$ if and only if $\chi$ itself
is afforded by a $K G$-module.
We observe that for any positive integer $n$,
$n\chi$ is the trace of a $K G$-module if and only if
$n$ is a multiple of $m_{\chi}$.

It is clear that a $KG$-module of trace $m_{\chi}\chi$ is irreducible.
Let $M$ be a $KG$-module of trace $m_{\chi}\chi$.
By Schur's Lemma, $D=\End_K(M)$
is a division algebra. The Schur index $m_{\chi}$ is the degree of $D$.

\begin{lemma} \label{Brauer-Speiser}
(Brauer~\cite{brauer}, Speiser~\cite{speiser})
If  $\chi$ is real-valued,  then $m_{\chi} \le 2$.
\end{lemma}

\begin{proof}{(Fein~\cite{fein}).}
Let $K$ be a totally real number field containing the values of $\chi$.
Let $m=m_{\chi}$.

Let $V$ be an irreducible $KG$-module whose trace is $m\chi$.
Set $D = \mathrm{End}_{KG}(V)$.
We are required to show that the degree of $D$ over $K$ is at
most~$2$.
In view of Lemma~\ref{ind-exp}, it suffices to show $D$ is isomorphic to its opposite.

However, this follows easily from the fact that $\chi$ is real-valued.
Let $V' = \Hom_{K}(V,K)$ be the vector space dual to $V$. Endow $V'$
with a $G$-action by declaring
$(g\cdot \lambda)(v) = g^{-1}(\lambda v)$.
The character of $V'$ is  $m\bar{\chi} = m\chi$ since $\chi$ is
real-valued.
Clearly, $\End_{KG}(V')$ is isomorphic to
the opposite algebra to $D$, and consequently
$D\simeq \End_{KG}(V)\simeq \End_{KG}(V') \simeq D^{\mathrm{opp}}$.
\end{proof}

\begin{lemma} \label{Schur index}
Let $\chi$ be an irreducible complex character of the finite group~$G$.
Let $K$ be an extension of $\QQ$ that contains the values of $\chi$.
If $\rho$ is a character of a $KG$-module, 
then $m_{\chi}| (\chi, \rho)$.
\end{lemma}

\begin{proof}
Let $M$ be a $KG$-module
affording the character $\rho$.
The $\chi$-isotypic part of $M$ is defined over $K$.
The lemma follows from the observation.
\end{proof}

\begin{corollary} \label{useful}
Let $H$ be a subgroup of the finite group $G$.
Let $\chi\in\Irr(G)$ be rational-valued.
Let $\rho$ be a character of $H$ corresponding to a representation
defined over $\QQ$.
If the integer $(\rho_H^G,\chi)$ is odd, then $\chi$
is the trace of a representation defined over $\QQ$.
\end{corollary}

\begin{proof}
This is a consequence of Lemmas
\ref{Brauer-Speiser} and~\ref{Schur index}.
\end{proof}

Let $K$ be a finite extension of $\QQ$ that contains the values of
$\chi$. Let $\MM$ be the set of places of $K$, including the infinite
places.
Since the completion $K_v$ contains $K$, hence a priori the values of
$\chi$,
it makes sense to compare the Schur index of $\chi$ over $K$
with the local Schur index of $\chi$ over $K_v$.

We record some basic facts.
See~\cite[Corollary 9.5]{feit82} for a proof of part (2).

\begin{lemma} \label{record}
Let $\chi$ be an irreducible complex character of the
finite group $G$.
\begin{enumerate}

\item The Schur index $m_{\chi}=1$ if and only if each of the
local Schur indices is equal to one. Somewhat more is true.
Let $v_0\in \MM$. If the local Schur index is one at each place $v\ne
v_0$ then so also is the Schur index at $v_0$ and the global Schur
index.

\item If $p$ is a prime and $p$ does not divide the order of $G$,
then the $p$-adic Schur index is~$1$.

\item The relation to the Frobenius-Schur indicator is as follows.
Let $K$ be the field generated over $\QQ$ by the values of $\chi$,
and let $k$ be the completion of $K$ at an infinite place.
The Frobenius-Schur indicator is $0$ if $k=\CC.$
The Frobenius-Schur indicator is $+1$ (resp.~$-1$)
if $k=\RR$ and the local Schur index at $\RR$ is~$1$
(resp.~$2$).
\end{enumerate}
\end{lemma}

\begin{lemma} \label{2-groups} Let $S$ be a \mbox{$2$-group}.
Suppose that $\chi$ in $\Irr(S)$ has rational values and
Frobenius-Schur indicator~$+1$. 
Then there is a $\QQ S$-module with
trace~$\chi$.
\end{lemma}

\begin{proof}
Since $\chi$ is real-valued,
the local Schur index of $\chi$ at the real place is $1$.
By Lemma~\ref{record}(2), for all odd  primes $p$,
the local Schur index at $p$ is $1$. Hence,
by Lemma~\ref{record}(1) with $v_0=2$, we see that
$m_{\chi}=1$. It follows that $\chi$ is afforded by a $KS$-module,
where $K$ is the field generated over $\QQ$ by the values of $\chi$.
But by hypothesis, $K=\QQ$.
\end{proof}

We also need the following:

\begin{lemma} \label{index 2} Let $N$ be a subgroup of
index~$2$ in the finite group $G$.  Let $V$
be an irreducible $\CC N$-module with character $\chi$.
Let $W$ be an irreducible
constituent of $V_N^G$.  If $\chi$ takes values
in $K$ and $W$ is defined over $K$, then $V$ is defined
over $K$.
\end{lemma}

\begin{proof}  Let $M=V_N^G$ be the induced module.
Either $M$ is irreducible or a
sum of two distinct irreducible modules.
In the former case, $M=W$, and in the latter case,
$M$ is the sum of $W$ and a twist of $W$.
In particular, since $W$ is defined over $K$,
so also is $M$.

If $M$ is not irreducible, then $M|_N \cong V$ is defined
over $K$ by assumption.
If $M$ is irreducible, then $(\chi_N^G, \chi_N^G)=1,$
and so, by Frobenius reciprocity and
by Lemma~\ref{Schur index}, $m_{\chi}$ divides $(\chi,(\chi_N^G)|_N)=1$.
Since the trace of $V$ takes values in~$K$,
it follows via Lemma~\ref{index 2} that $V$ is defined over~$K$.
\end{proof}

\section{Classical Weyl groups} \label{classical}
We remark that there are many proofs that the symmetric group $\Sn$ 
is split. For one such see~\cite{james}. 

Let $G$ be a finite group.
We note that all of the
irreducible complex characters of $G$ are rational-valued if and only if,
for all $g\in G$ and all integers $r$ coprime to $|G|$, there
exists $y \in G$ such that $g^r=g^y$.
Similarly, all of the irreducible complex characters of $G$ are
real-valued if and only if $g$ is conjugate to
$g^{-1}$ for all $g\in G$.
We will need the following lemma.

\begin{lemma} \label{conjugacy}
Let $G$ be a finite group all of whose irreducible complex characters
are rational-valued. Let $G_0<G$ be a subgroup of index~$2$.
Let $g\in G$, and let $r$ be an integer coprime to $|G|$.
If $G_0$ does not contain the centralizer in $G$ of $g$,
then $g^r=g^y$ for some $y\in G_0$.
\end{lemma}

\begin{proof}
Since all of the irreducible complex characters are rational-valued,
there exists $y$ in $G$ such that
$g^r=g^y$. By hypothesis, there exists $u\in G\setminus G_0$
that centralizes $g$. Then both $y$ and $uy$
conjugate $g$ to $g^r$, and precisely one of these lies in $G_0$.
\end{proof}

For $X$ a group, let $X\wr (\ZZ/2\ZZ)$
denote the group generated by $X\times X$ and its automorphism
$\tau(x,y)=(y,x)$.

\begin{lemma} \label{wreathlemma}
Let $X$ be a finite group with all representations
defined over~$\QQ$.
\begin{enumerate}
\item The group $Y=X\wr(\ZZ/2\ZZ)$ is split.
\item
Let $X_0<X$ be a subgroup
of index~$2$, and suppose that $X_0$ is split.
Let $W_0\le X\times X$ be the subgroup consisting of those
pairs $(x,y)$ such that $xy\in X_0$. Then $W= \langle W_0, \tau\rangle$
is split.
\end{enumerate}
\end{lemma}

\begin{proof}
Statement (1) is the assertion of Theorem~\ref{wreath2} for $K=\QQ$ and
$n=2$, and we turn to (2).

We claim that all irreducible complex characters of $W_0$ are
rational-valued.
Indeed, this is the case for $X_0\times X_0$,
and if $w \in W_0\setminus(X_0\times X_0),$
then the centralizer in ${X\times X}$ of $w$ 
is not contained in $W_0$, and applying
Lemma~\ref{conjugacy} to $W_0 \le X\times X$ proves the claim.

Note that $W/(X_0 \times X_0)$ is the Klein four group, hence
$W$ is the union of three proper subgroups of index~$2$, namely,
$W_0, X_0 \wr(\ZZ/2\ZZ)$, and a subgroup conjugate to
$X_0 \wr(\ZZ/2\ZZ)$ in $Y$.
Since each of these three subgroups has all irreducible
complex characters rational-valued, the same is true for $W$, by the
well-known criterion:
all irreducible complex characters of
the finite group $G$ are rational-valued if and only if
for all $g\in G$, and all integers $r$ coprime to $|G|$, $g\sim g^r$.

Part (2) now follows from Lemma~\ref{index 2} applied to
the index-2 inclusion \mbox{$W\le Y$}.
\end{proof}

\begin{lemma}  \label{sym}
Let $E$ be an elementary abelian $2$-subgroup
of $\Sn$.  Let $C$ be the centralizer of $E$ in $\Sn$, and
let $S$ be a Sylow $2$-subgroup of~$C$.
\begin{enumerate}
\item $C$ is split.
\item If $E$ has no fixed points, then $C \cap \An$ is split.
\item $S$ is split.
\item $S \cap \An$ is split.
\end{enumerate}
\end{lemma}

\begin{proof}
Let $\Omega$ denote the set of cardinality $n$ on which $\Sn$ acts.

We prove (1) and (3) together.
If $n=1$, then $C=S=1$ is split.
We proceed by induction on $n$.

Assume that $E$ does not act regularly on some $E$-orbit $\Omega$.
Then there is a partition $\Omega = \Omega_1 \coprod \Omega_2$, 
such that the kernel of the action of
$E$ on $\Omega_1$ has no fixed points on $\Omega_2$.
Thus each $\Omega_i$ is invariant under the centralizer $C$ of $E$.
Let $C_i$ be the centralizer in $\Sym(\Omega_i)$
of the image of $E$.
Then $C = C_1 \times C_2$ and $S  = S_1\times S_2$,
where $S_i$ is the Sylow $2$-subgroup of the centralizer of an elementary
abelian $2$-group in a smaller symmetric group.
Now $C$ and $S$ are split by the 
inductive hypothesis and Lemma~\ref{products}.

Suppose $E$ acts regularly on all $E$-orbits.  
Then $C = E \wr \Sym_d$
where  $n=d|E|$ and so $C$ is split by Theorem \ref{wreath2}.

Similarly, if $S$ is not transitive on $\Omega$, then 
we can write  $\Omega = \Omega_1 \coprod \Omega_2$, where each
$\Omega_i$ is $S$-invariant. As above, (3) follows by
induction.

Suppose $S$ is transitive on $\Omega$.
Then  $S = E \wr S_0$ where $S_0$ is a Sylow $2$-subgroup
of a symmetric group of degree $m$, where
$m\abs{E} = \abs{\Omega}$. Necessarily $m$ is
a power of $2$.
Now $S$ is split by Corollary~\ref{sylowwreath}
applied to $E=G$.
This finishes the proof of (1) and (3) and we turn now to~(2).

Assume $E$ has no fixed points, and set 
$D=C\cap \An$.
We claim that if the integer $r$ is coprime to the order of $g \in
C$, then there exists $d\in D$ such that $g^d = g^r$.
Clearly, the claim, if true, would imply (2) by Lemma~\ref{conjugacy}.

As above, we are done by induction if some orbit of $E$ is not
regular.  So we assume that is the case.
Note that if $\abs{E}\ge 4$, then $C\le \An$ and the claim follows
from (1). 

Assume $E$ of order $2$ acts semiregularly. The nonidentity element
of $E$ then is a fixed-point-free involution.
Any $c\in C$ can be written 
$$
c=\sigma(e_1,\dots e_m)
$$
where $2m=n, e_i\in E$, and $\sigma\in\Sym_m$.
If $\sigma$ is not a $m$-cycle, the claim follows
by induction.

If $e_1e_2\cdots e_m \ne 1$, then $\sigma \notin D$
and the claim follows from Lemma~\ref{wreathlemma}(1).

If $e_1e_2\cdots e_m  = 1$, then $c$ is conjugate to $\sigma$.
The claim now follows as $\sigma$ and $\sigma^r$
are conjugate in $\Sym_m\le \An$.
This finishes the proof of the claim and of~(2).

We prove (4). Let $T = S \cap \An$.
Since $S$ is split and $|S:T|=2$, it suffices
to show that all irreducible complex characters of $T$
are rational-valued. We will prove the equivalent property:
if $s \in S$ and $r$ is odd, then $s^r=s^g$ for some $g \in T$.

As above, if $S$ is not transitive, the result follows by induction.
If $S$ is transitive, then all orbits of $E$ are regular.
Lemma~\ref{wreathlemma}(2) shows that $T$ is split.
\end{proof}

Our main result for the classical
Weyl groups now follows.

\begin{theorem}  \label{classicalthm}
Let $E$ be an elementary abelian subgroup
of a classical Weyl group~$W$. Let $C$ be the centralizer of $E$
in $W$, and $T$ a Sylow $2$-subgroup of $C$. Then both $C$ and $T$
are split.
\end{theorem}

\begin{proof} For $W$ of type $A_n$, this is equivalent to assertions
(1) and (3) of Lemma~\ref{sym}.
The case of type $B_n$ is a
special case of the type $A_{2n}$ result.
(Indeed, since $W$ of type $B_n$ is the centralizer of a fixed
point free involution $z$ in $\Sym_{2n}$, it follows that
the centralizer in $W$ of $E$ equals the centralizer in $\Sym_{2n}$
of the elementary abelian $2$-subgroup $\langle z, E\rangle$.)
For $W$ of type $D_n, n$ odd, this follows from
Lemma~\ref{products},
since $W \times (\ZZ/2\ZZ)$ is a Weyl group of type $B_n$.
For $W$ of type $D_n,n$ even,
this follows from Lemma~\ref{sym}(2) and (4) since
$W = B_n \cap \mathrm{Alt}_{2n}$.
\end{proof}

\section{Exceptional Weyl groups} \label{exceptional}

For the exceptional Weyl groups, we prove somewhat more
than Theorem~\ref{main}.

\begin{theorem} \label{gg3}
Let $W$ be the Weyl group of an exceptional complex simple Lie algebra
($G_2$, $F_4$, $E_6$, $E_7$, or  $E_8$).
Let $C$ be the centralizer of an elementary abelian $2$-subgroup of
$W$. Let $T$ be a Sylow $2$-subgroup of $C$.
\begin{enumerate}
\item All characters of $T$ are rational-valued.
All Frobenius-Schur indicators of $T$ are equal to $+1$.
Consequently, $T$ is split.
\item For every $\chi\in Irr(C)$, there exists
$\psi\in \Irr(T)$ such that $(\chi,\psi_T^C)$ is odd.
\item All irreducible characters of $C$ are rational-valued.
\end{enumerate}
\end{theorem}

The proof of Theorem~\ref{gg3} is a direct calculation in
Magma~\cite{magma}
(see~\cite[Appendix A]{gg}).
Now Theorem~\ref{main} follows from Theorem~\ref{gg3},
the Brauer-Speiser Theorem and Lemma~\ref{Schur index}.

We have already seen that part~(1) of Theorem~\ref{gg3}
holds for the classical Weyl groups.
We have verified part (2) for $\Sn, 1\le n\le 25$
by a direct calculation in Magma
(see~\cite[Appendix B]{gg}), and we we suspect that it holds for all
classical Weyl groups.

\section{Positive characteristic} \label{positive}
In order to apply our result to representations in positive
characteristic, we appeal to
a result of Thompson~\cite[p.\ 327]{jgt}.
Let $G$ be a finite group, $F$ an algebraically closed
field of positive characteristic~$p$, and $V$ an irreducible $FG$-module.
Thompson's result states that if $p$ is odd and $V$ is self-dual,
then there is an irreducible self-dual characteristic zero
representation $W$ of $G$, such that 
$V$ appears with odd multiplicity as a composition factor 
of the reduction mod $p$ of $W$.
Moreover, any such self-dual $W$ has the
same Frobenius-Schur indicator as that of $V$.

\begin{corollary}[Corollary to Thompson's result] \label{tc}
Let $G$ be a finite group all of whose irreducible complex
representations can be defined over~$\QQ$.
Let $F$ be an algebraically closed field of positive characteristic~$p$.
Then every irreducible $FG$-module $V$ is self-dual and can be
defined over the prime field.
Furthermore, if $p$ is odd, then
the Frobenius-Schur indicator of $V$ is~$+1$.
\end{corollary}

\begin{proof}
Since $G$ is split, all characters of $G$ are real-valued,
and so every element of $G$ is conjugate to its inverse.
It follows that $V$ is self-dual.

Since $g^r\sim g$ for $r$ coprime to $|G|$, it follows that
the trace of $g$ on $V$ has  values in the prime field.
It now follows from Wedderburn's theorem that $V$
can be defined over the prime field.

Suppose $p$ is odd, and let $W$ be a characteristic zero irreducible module 
such that $V$ occurs in the reduction of $W$ modulo $p$ with odd multiplicity
(such a $W$ exists by Thompson's result).
Since $G$ is split, $W$ has Frobenius-Schur indicator
$+1$, whence $V$ has Frobenius-Schur indicator $+1$, as was to be shown.
\end{proof}

As a consequence of Corollary~\ref{tc} and Theorem~\ref{main}, we have:

\begin{corollary} \label{usingthompson} Let $W$ be a Weyl group
and $E$ an elementary abelian $2$-subgroup of $W$.
Let $C$ be the centralizer of $E$ in $W$, and let $T$
be a Sylow $2$-subgroup of $C$. Let $R=C$ or $T$.
If $F$ is an algebraically closed
field of odd characteristic, then
every irreducible $FR$-module is self-dual,
can be realized over the prime field, and has Frobenius-Schur indicator~$+1$.
\end{corollary}

\section{Real reflection groups} \label{reflection}
A transformation $g\in GL_n(\CC)$ is a \mathdef{real reflection}
if $\dim\ker(g-I_n)=n-1$ and $\dim\ker(g+I_n)=1$.
A finite subgroup of $GL_n(\CC)$ that is generated by
real reflections is a \mathdef{real reflection group}.
This class includes the finite Weyl groups.
In the context of this paper, it is natural to ask
which real reflection groups are split.

Since every real reflection group is a direct
product of irreducible real reflection groups, and
in view of Theorem~\ref{direct products}, we will only
consider irreducible reflection groups.

The irreducible real reflection groups that
are not Weyl groups are $H_3,H_4$ and
the dihedral groups $I_n$ of order $2n$ for $n=5$ and $n\ge7$.

Let $G=H_3$ or $H_4$. Then $G$ is not split.
The centralizer of a non-central elementary
abelian $2$-subgroup of $G$ is split. The centralizer of any elementary
abelian $2$-subgroup of $G$  has a Sylow $2$-subgroup that is split.

The group $I_n$ is split if and only if $n\le 4$.
A non-central elementary abelian $2$-subgroup
is self-centralizing, has order~$4$ and is split.
A Sylow $2$-subgroup of $I_n$ is split
if and only if $n$ is not divisible by $8$.

\section{Open questions} \label{open questions}
We proved Theorem~\ref{gg3}(2) for the exceptional Weyl groups only.

\begin{question} \label{cq}
Does the conclusion of Theorem~\ref{gg3}(2)
hold for the classical Weyl groups?
\end{question}

\begin{question}
Is it true that a Sylow $2$-subgroup of a split finite group is split?
\end{question}

Let $G$ denote an arbitrary finite group and
$S$ a Sylow $2$-subgroup of $G$.
As we have observed, a $2$-group is split
if and only if all characters have values in $\QQ$
and all Frobenius-Schur indicators are $+1$.

Thus, the previous question divides naturally into two parts.

\begin{question} If all irreducible characters of $G$
have Frobenius-Schur indicators $+1$, is the same true for~$S$?
\end{question}

\begin{question} If all irreducible characters of $G$ are rational-valued
is the same true for~$S$?
\end{question}

We also ask:
\begin{question} If all irreducible characters of $G$ have
Frobenius-Schur indicator equal to~$+1$, and all irreducible characters of $S$ are
real-valued, then do all irreducible characters of $S$
have Frobenius-Schur indicator equal to~$+1$?
\end{question}

\appendix

\section{Code} \label{code}
\begin{verbatim}
QQ := Rationals();

// Definition. The complex representation of the finite
// group X is *split*  if it can be defined over QQ.
// X is *split* if all of its (irreducible) complex
// representations are split. A necessary condition
// for X to be split is that it has *rational character
// values* , i.e. if all its character values lie in QQ.
//
// Note that the converse holds if X is a 2-group.


// Given the character table ct of a finite group g,
// "blah(ct)" returns TRUE if all character values of
// g are rational and if all the Frobenius-Schur
// indicators are 1, and FALSE otherwise.
/////////////////////////////////////////////////
function blah(ct)
 b1 :=  &and[ct[i][j] in QQ: i,j in [1..#ct]];
 b2 :=  &and[Schur(j,2) eq 1: j in ct];
 return b1 and b2;
end function;

// On input subsets A and B
// of a finite group g, such
// that no two elements of
// A are conjugate in g,
// "dedupe(g,A,B)" outputs a
// maximal subsetC of AUB such
// that no two elements of C
// are conjugate in g.
//////////////////////////////
function dedupe(g,A,B)
 a := #A;
 b := #B;
 C := Setseq(A) cat Setseq(B);
 flag := [C[j+a] notin A : j in [1..b]];
 for i in [1..a+b], j in [1..b] do
    if i lt j+a and flag[j] and IsConjugate(g,C[i],C[j+a]) then
       flag[j] := false;
    end if;
 end for;
 return {C[j+a] : j in [1..b] | flag[j]};
end function;

// On input a finite group X,
// "bs(X)" tries to show that X is split using
// the Brauer-Speiser theorem
// 1. Suppose X is a 2-group. Outputs TRUE if X has
//    rational character values and FALSE otherwise.
// 2. Suppose X is not a 2-group. Let s be a Sylow
//    2-subgroup of X.
//    a. Construct all (classes of)
//       irreducible representations of s.
//    b. Induce these up to X, to get a set I of
//       split reps of X. Try to prove X split
//       by using Brauer-Speiser.
//    c. Outputs (b1,b2,b3):
//       b1 if X has rational character values
//       b2 if s is split
//       b3 if the process described in b. above
//       constructs all reps of X, hence proving that
//       X is split.
//    d. If b3 TRUE, then stop. Else
//    e. Take tensor products.
//    f. Try to prove X split by using Brauer-Speiser:
//       if chi \in Irr(X) and (chi,i) odd for i \in I or i \in
//       I\tensor I then chi is split.
//    g. Output b4 if taking tensor products constructs all
//       reps of X. Stop.
//////////////////////////////////////////////
procedure bs(X)
 start := Cputime();
 s := SylowSubgroup(X,2);
 ctX := CharacterTable(X);
 bool := [blah(ctX)];
 if #X eq #s then
    <#X, Cputime()-start, bool>; return;
 end if;
 cts := CharacterTable(s);
 Append(~bool, blah(cts));
 rat := {j : j in cts | Schur(j,2) eq 1 and
         &and[j[i] in QQ : i in [1..#cts]]};
 ind := [Induction(j,X) : j in rat];

 S := {};
 out := {j : j in ctX};
 for r in ind do
   for j in out do
     if IsOdd(Integers()!InnerProduct(r,j)) then
        Include(~S,j);
        Exclude(~out,j);
     end if;
   end for;
 end for;
 b := #S eq #ctX;
 Append(~bool, b);
 if b then
    <#X, Cputime()-start, bool>; return;
 end if;

 for i,j in S do
   for t in out do
     if IsOdd(Integers()!InnerProduct(t,i*j)) then
        Include(~S,t);
        Exclude(~out,t);
     end if;
   end for;
 end for;

 Append(~bool,  #S eq #ctX);
 <#X, Cputime()-start, bool>;
end procedure;


// On input the finite group W,
// "get_cents_of_el_ab_twos(W)"
// outputs the set of all centralizers
// of elementary abelian 2-subgroups
// of W.
//////////////////////////////////////
function get_cents_of_el_ab_twos(W)
  old := {};
  current := {W};
  repeat
    start := Cputime();
    new := {};
    for g in current do
      if IsAbelian(g) then continue; end if;
      cc := [j[3] : j in ConjugacyClasses(g)];
      c := Center(g);
      inv := [j : j in cc | Order(j) eq 2 and j notin c];
      cents := {Centralizer(g,j) : j in inv};
      new := new join dedupe(W,new,cents);
    end for;
    #old, Cputime()-start;
    old := old join dedupe(W,old,current);
    current := new;
  until #new eq 0;
  return old;
end function;

for D in ["G2", "H3", "H4", "F4", "E6", "E7", "E8"] do
  D;
  W := CoxeterGroup(D);
  cents := get_cents_of_el_ab_twos(W);
  {* #j : j in cents *};
  for j in cents do
    bs(j);
  end for;
end for;
\end{verbatim}

\section{More code} \label{morecode}
\begin{verbatim}
// On input n, "sym(n)"
// constructs the matrix
// with rows Irr(S_n)
// and columns Irr(s),
// for s a Sylow 2-subgroup,
// end entries the mod 2
// inner product with the
// induced character.
///////////////////////////////////
function sym(n)
 Pa := Partitions(n);
 charsg := [SymmetricCharacter(pa) : pa in Pa];
 g := Sym(n);
 s := SylowSubgroup(g,2);
 cts := CharacterTable(s);
 res := [Restriction(j,s) : j in charsg];
 return Matrix([[Integers()!InnerProduct(j, x) mod 2 :
        j in cts]: x in res]);
end function;

// On input a matrix M
// "oneineachrow(M)"
// outputs TRUE if there is at
// least one 1 in each row,
// and FALSE otherwise.
///////////////////////////////////
function oneineachrow(M);
 c := Ncols(M);
 r := Nrows(M);
 oneinthisrow := [ &or[M[i][j] eq 1: j in [1..c]]: i in [1..r]];
 return &and oneinthisrow;
end function;

for n in [1..500] do
  start := Cputime();
  M := sym(n);
  n, oneineachrow(M), Cputime()-start;
end for;

quit;
\end{verbatim}
\end{document}